\documentclass{article}

\usepackage[latin1]{inputenc}
\usepackage{amssymb,amsmath,amsthm,amscd,mathrsfs}
\usepackage[all]{xy}

\newtheorem{thm}{Theorem}[section]
\newtheorem{defi}[thm]{Definition}
\newtheorem{prop}[thm]{Proposition}
\newtheorem{lemme}[thm]{Lemma}
\newtheorem{cor}[thm]{Corollary}

\DeclareMathOperator{\Prym}{Prym}
\DeclareMathOperator{\Pic}{Pic}

\DeclareMathOperator{\sign}{sign}
\DeclareMathOperator{\Sing}{Sing}
\DeclareMathOperator{\Div}{Div}
\DeclareMathOperator{\diag}{diag}

\newcommand{\eq}[1][r]
{\ar@<-3pt>@{-}[#1]
\ar@<-1pt>@{}[#1]|<{}="gauche"
\ar@<+0pt>@{}[#1]|-{}="milieu"
\ar@<+1pt>@{}[#1]|>{}="droite"
\ar@/^2pt/@{-}"gauche";"milieu"
\ar@/_2pt/@{-}"milieu";"droite"}

\newcommand{\incl}[1][r]
{\ar@<-0.2pc>@{^(-}[#1] \ar@<+0.2pc>@{-}[#1]}

\begin{document}

\title{\bf Duality for relative Prymians associated to K3 double covers of del Pezzo surfaces of degree 2}

\author{Grégoire \textsc{Menet}}
\maketitle

\begin{abstract}
Markushevich and Tikhomirov provided a construction of an irreducible symplectic V-manifold of dimension 4, the relative compactified Prym variety of a family of curves with involution, which is a Lagrangian fibration with polarization of type (1,2). We give a characterization of the dual Lagrangian fibration. We also identify the moduli space of Lagrangian fibrations of this type and show that the duality defines a rational involution on it.
\end{abstract}

\section*{Introduction}
Irreducible symplectic varieties are defined as compact holomorphically symplectic Kähler varieties with trivial fundamental group, whose symplectic structure is unique up to proportionality.

By the Bogomolov decomposition Theorem \cite{Bogo}, irreducible symplectic varieties play (together with Calabi Yau manifolds and complex tori) a central role in the classification of compact Kähler manifolds with torsion $c_{1}$.

Very few deformation classes of irreducible symplectic varieties are known.
For any positive integer $n$, Beauville exhibited 2 examples of dimension $2n$ in \cite{Beauville}: the Hilbert scheme $X^{[n]}$ parametrizing 0 dimensional subschemes of length $n$ on a K3 surface $X$, and the generalized Kummer variety $K^{n}(T)$ of a 2-dimensional torus $T$, namely the locus in $T^{[n+1]}$ parametrizing the subschemes whose associated cycle sums up to 0 in $T$.

Besides the Beauville examples, there are only two more known constructions of irreducible symplectic varieties up to deformation equivalence, produced by O'Grady in \cite{Grady1} and \cite{Grady2}, and their dimensions are respectively ten and six.

The problem of extending the very short list of known deformation classes of irreducible symplectic varieties is very hard. One can obtain a larger stock of examples if one turns back to the original setting of Fujiki (see \cite{Fujiki}), who considered symplectic V-manifolds. A V-manifold is an algebraic variety with at worst finite quotient singularities. We will say that a V-manifold is symplectic if its nonsingular locus is endowed with an everywhere nondegenerate holomorphic 2-form. We will say moreover that a symplectic V-manifold is irreducible if it is complete, simply connected, and if the holomorphic 2-form is unique up to $\mathbb{C}^*$. Such varieties should appear as factors in the generalized Bogomolov decomposition conjecture (see \cite{Kata} and \cite{Mo}). All the examples of symplectic V-manifolds given by Fujiki, up to deformations of complex structure, are partial resolutions of finite quotients of the products of two symplectic surfaces. Markushevich and Tikhomirov provide in \cite{Markou} a new construction of an irreducible symplectic V-manifold $\mathcal{P}$ of dimension 4, the relative compactified Prym variety of some family of curves with involution. The structure map of the relative compactified Prym variety is a Lagrangian fibration which has a (1,2)-polarized Prym surface as generic fiber.
A natural question is what the dual of this (1,2)-polarized fibration is.

The irreducible symplectic varieties and V-manifolds with a Lagrangian fibration are of particular interest, as they generalize K3 surfaces with elliptic pencils on one hand, and the phase spaces of algebraically integrable systems on the other hand. The problem of constructing the dual of a Lagrangian fibration is discussed in \cite{Sawon}, where an interesting link to the twisted Fourier-Mukai transform is uncovered.

The construction of $\mathcal{P}$ starts from a pair of totally tangent plane quartics $\overline{B_{0}}$ and $\overline{\Delta_{0}}$. The first is used to construct a degree 2 del Pezzo surface $X$, and the second determines a K3 double cover $S$ of $X$.
Then the wanted family of curves is a non-complete linear system of curves on $S$, and $\mathcal{P}$ is its relative compactified Prymian. Permuting the roles of $\overline{B_{0}}$, $\overline{\Delta_{0}}$, we obtain another K3 surface $\widetilde{S}$ and the corresponding Prymian $\widetilde{\mathcal{P}}$.

We will prove that the Lagrangian fibrations on $\mathcal{P}$ and $\widetilde{\mathcal{P}}$ are dual to each other. Moreover we will prove that not only $S\not\simeq\widetilde{S}$
for generic $S$, but also that the derived categories of $\widetilde{S}$, $S$ are non-equivalent and $\widetilde{S}^{[2]}\not\simeq S^{[2]}$. This will allow us to conclude that the associated compactified Prymians $\widetilde{\mathcal{P}}$, $\mathcal{P}$ are non-isomorphic.

In the first section we will recall the construction of Markushevich and Tikhomirov \cite{Markou}. In the second section, we will give a characterization of the dual of a (1,2)-polarized Prym surface following Barth \cite{Barth}. In the third section, we will use results of Yoshikawa \cite{Yoshikawa} on moduli of 2-elementary K3 surfaces which will allow us to identify the moduli space $\mathfrak{P}$ of compactified Prymians $\mathcal{P}$ and to conclude that, in the generic case, the "dual" K3 surfaces $S$, $\widetilde{S}$ are derived equivalent if and only if they are isomorphic. And in the fourth section we will finally prove that the dual Prymian $\widetilde{\mathcal{P}}$ is generically non-isomorphic to $\mathcal{P}$, so that the duality is indeed a non-trivial involution on $\mathfrak{P}$.
\newline

\textbf{Acknowledgement.} I would like to thank Dimitri Markushevich for his help.

\section{Construction of $\mathcal{P}$}\label{Rappels}
In this section we follow \cite{Markou}. We will start by the construction of the family of genus 3 curves with involution, whose relative compactified Prymian $\mathcal{P}$ is the irreducible symplectic V-manifold discussed in the introduction.

Let $\overline{B_{0}}$ be a smooth quartic curve in $\mathbb{P}^{2}$. Let $\mu:X\rightarrow \mathbb{P}^{2}$ be the double cover branched in $\overline{B_{0}}$. Then $X$ is a Del Pezzo surface of degree 2. Let $\overline{\Delta_{0}}$ be a smooth quartic curve in $\mathbb{P}^{2}$
totally tangent to $\overline{B_{0}}$ at eight distinct points. Then the linear pencil $\left\langle \overline{B_{0}}, \overline{\Delta_{0}}\right\rangle$ contains a double conic, so the eight tangency points of $\overline{B_{0}}$, $\overline{\Delta_{0}}$ lie on a conic.
Let $B_{0}=\mu^{-1}(\overline{B_{0}})$. We have $\mu^{-1}(\overline{\Delta_{0}})=\Delta_{0}+i(\Delta_{0})$, where $\Delta_{0}$ is a smooth curve.
By Lemma 5.14 of \cite{Shouhei}, $\Delta_{0}\in|-2K_{X}|$.
Finally, let $\rho:S\rightarrow X$ be the double cover branched in $\Delta_{0}$,  $\Delta=\rho^{-1}(\Delta_{0})$. Note that if we take a similar double cover branched in $i(\Delta_{0})$, $\rho':S'\rightarrow X$, we get a surface $S'$ isomorphic to $S$ (indeed, $i\circ\rho'$ and $\rho$ are two double covers branched in the same curve in $X$).
Denote by $\tau$ the involution of $S$ induced by $\rho$. We have the following diagram:

$$\xymatrix{
 & & B_{0} \ar@{=}[rr] \ar@{^{(}->}[d]& & \overline{B_{0}}\ar@{^{(}->}[d] & & \\
S \ar@(ul,dl)[]_{\tau}\ar[rr]_\rho^{2:1} & & X \ar[rr]_\mu^{2:1} & & \mathbb{P}^{2} & & (1).\\
 \Delta \ar@{=}[rr] \ar@{^{(}->}[u] & & \Delta_{0} \ar@{^{(}->}[u] & & & & }$$

We also allow the case where
$\overline{B_{0}}$ is a quartic and $\overline{\Delta_{0}}$ is equal to a double conic $2Q$ such that $\Delta_{0}=\mu^{-1}(Q)$ is a smooth curve. In this case $i(\Delta_{0})=\Delta_{0}$ and $\Delta_{0}$ is in $|-2K_{X}|$. We will have some additional conditions for matching with \cite{Markou}.
Define also $H=\rho^{*}(-K_{X})$.

The involution $\tau$ of the double cover $\rho: S\rightarrow X$ is $H$-linear and induces an involution on $|H|\simeq \mathbb{P}^{3}$, whose fixed locus consists of two components: a point and a plane. The plane parametrizes the curves of the form $\rho^{-1}\mu^{-1}(t)$, where $t$ is a line in $\mathbb{P}^{2}$. Thus this plane is parametrized by the dual of $\mathbb{P}^{2}$, denoted $\mathbb{P}^{2\vee}$. Let $\epsilon:\mathcal{C}\rightarrow \mathbb{P}^{2\vee}$ be the linear subsystem of $\tau$-invariant curves parametrized by $\mathbb{P}^{2\vee}$. The properties of this linear subsystem must be as in \cite{Markou}.

Then, we need the following conditions. The plane quartic $\overline{B_{0}}$ must have exactly 28 bitangent lines $m_{1},...,m_{28}$. The curve $\mu^{-1}(m_{i})$ is the union of two rational curves $l_{i}\cup l_{i}'$ meeting in 2 points. The 56 curves $l_{i}, l_{i}'$ are all the lines on $X$, that is, curves of degree 1 with respect to $-K_{X}$. Further, the curves $C_{i}=\rho^{-1}(l_{i})$, $C_{i}'=\rho^{-1}(l_{i}')$ are plane conics on $S$ with respect to the injection $S\hookrightarrow \mathbb{P}^{3}$ defined by $|H|$. The conics $C_{i}$, $C_{i}'$ must be irreducible and meeting in exactly 4 distinct points. 

Consequently we require the following conditions for the couple $(\overline{B_{0}},\overline{\Delta_{0}})$. 
\begin{defi}\label{gene}
A pair $(\overline{B_{0}},\overline{\Delta_{0}})$ will be called sufficiently generic if the following conditions are satisfied:
\begin{itemize}
\item 
The quartic $\overline{B_{0}}$ must not have a tangent line with multiplicity 4 in a point. 
In this case $\overline{B_{0}}$ has exactly 28 bitangent lines.
\item
A bitangent line of $\overline{B_{0}}$ tangent at $\overline{B_{0}}$ in a point $p$ must not be tangent at  $\overline{\Delta_{0}}$ in this same point $p$.
In this case, the conics $C_{i}$, $C_{i}'$ are not tangent, so meet in exactly 4 distinct points.
\item
The quartics $\overline{B_{0}}$ and $\overline{\Delta_{0}}$ must not have a common bitangent line.
In this case the conics $C_{i}$ and $C_{i}'$ are irreducible. Moreover $S$ contains no lines.
\end{itemize}
We will denote the set of sufficiently generic pairs $(\overline{B_{0}},\overline{\Delta_{0}})$ by $\mathfrak{L}$.
\end{defi}
Assume for the rest of the section that $(\overline{B_{0}},\overline{\Delta_{0}})\in \mathfrak{L}$.

Let $\mathcal{M}=M_{S}^{H,s}(0,H,-2)$ be the moduli space of semistable sheaves $\mathcal{F}$ on $S$ with respect to the ample class $H$ with Mukai vector $v(\mathcal{F})=(0,H,-2)$. We define an involution on $\mathcal{M}$ by
$$\sigma:\mathcal{M}\rightarrow\mathcal{M},\ [\mathcal{L}]\mapsto [\mathcal{E}xt_{\mathcal{O}_{S}}^1(\mathcal{L},\mathcal{O}_{S}(-H))],$$
and we set $\kappa=\tau^{*}\circ\sigma$.
One can prove that $\kappa$ is a regular involution on $\mathcal{M}$ and that its fixed locus has one 4-dimensional irreducible component plus 64 isolated points.
\begin{defi}\label{Prym}
We define the compactified Prymian $\mathcal{P}$ as the 4-dimensional component of $Fix(\kappa)$.
\end{defi}

\begin{thm}
The variety $\mathcal{P}$ is an irreducible symplectic V-manifold of dimension 4 with only 28 singular points analytically equivalent to 
$(\mathbb{C}^4 / \left\{\pm1\right\},$ $0)$.
\end{thm}
\begin{proof}
See Theorem 3.4, Proposition 5.4 and Corollary 5.7 of \cite{Markou}.
\end{proof}

Now, we will introduce the Lagrangian fibration. We consider the linear subsystem $\epsilon:\mathcal{C}\rightarrow \mathbb{P}^{2\vee}$. If $t\in\mathbb{P}^{2\vee}$ is not tangent to $\overline{B_{0}}$ neither to $\overline{\Delta_{0}}$, which is the generic case,
then $C_{t}=\epsilon^{-1}(t)=\rho^{-1}\mu^{-1}(t)$ is a smooth genus-3 curve, and $E_{t}=C_{t}/\tau$ is elliptic. The double cover $\rho_{t}=\rho_{|C_{t}}:C_{t}\rightarrow E_{t}$ is branched at 4 points of the intersection $\Delta_{0}\cap E_{t}$ and the double cover $\mu_{t}=\mu_{|E_{t}}:E_{t}\rightarrow t\simeq \mathbb{P}^{1}$ is branched at 4 points of the intersection $\overline{B_{0}}\cap t$. We denote also $\tau_{t}=\tau_{|C_{t}}$.
Thus, we have the tower of double covers:
$$C_{t}\stackrel{2:1}{\rightarrow}E_{t}\stackrel{2:1}{\rightarrow}\mathbb{P}^{1}.$$
The following Lemma introduces the (1,2)-polarized Prym surfaces:

\begin{lemme}\label{prym}
For a generic line $t\in\mathbb{P}^{2}$, $ker(id+\tau_{t})$ has only one connected component in $J(C_{t})$, and the restriction of the principal polarization from $J(C_{t})$ to the abelian variety $\Prym(C_{t},\tau_{t})=ker(id+\tau_{t})$ is a polarization of type $(1,2)$.
\end{lemme}
\begin{proof}
See Lemma 3.2. in \cite{Markou}.
\end{proof} 
\begin{thm}
Identifying, as above, the 2-dimensional linear subsystem of $\tau$-invariant curves in $|H|$ with $\mathbb{P}^{2\vee}$, let $f_{\mathcal{P}}: \mathcal{P}\rightarrow \mathbb{P}^{2\vee}$ be the map sending each sheaf to its support. Then $f_{\mathcal{P}}$ is a Lagrangian fibration on $\mathcal{P}$ and the generic fiber $f_{\mathcal{P}}^{-1}(t)$ is the (1,2)-polarized Prym surface $\Prym(C_{t},\tau_{t})$.
\end{thm}
\begin{proof}
See Theorem 3.4 of \cite{Markou}.
\end{proof}

In fact, $\mathcal{P}$ is bimeromorphic to a partial resolution of a quotient of $S^{[2]}$.
Consider Beauville's involution (see Section 6 of \cite{Beau}):$$\iota_{0}: S^{[2]}\rightarrow S^{[2]}, \xi\mapsto\xi'=(\left\langle \xi\right\rangle\cap S)-\xi.$$
Here $S$ is taken in its embedding as a quartic surface in $\mathbb{P}^3$, given by the linear system $|H|$, $\left\langle \xi\right\rangle$ stands for the line in $\mathbb{P}^3$ spanned by $\xi$, and $\xi'$ is the residual intersection of $\left\langle \xi\right\rangle$ with $S$. By \cite{Beau}, this involution is regular whenever $S$ contains no lines, which is true in our case. Further, $\tau$ induces on $S^{[2]}$ an involution which we will denote by the same symbol. As $\tau$ on $S$ is the restriction of a linear involution on $\mathbb{P}^3$, $\iota_{0}$ commutes with $\tau$, and the composition $\iota=\iota_{0}\circ\tau$ is also an involution.
\begin{prop}
The fixed locus of $\iota$ is the union of a K3 surface $\Sigma\subset S^{[2]}$ and 28 isolated points.
\end{prop}
\begin{proof}
See Lemma 5.3 of \cite{Markou}. In fact, as follows from the recent work of Mongardi \cite{Mongordi}, the fixed locus of any symplectic involution on an irreducible symplectic variety deformation equivalent to the Hilbert square of a K3 surface is as in the statement of the proposition.
\end{proof}
Let $M=S^{[2]}/\iota$ and $\overline{\Sigma}$ be the image of $\Sigma$ in $M$. We also denote by $M'$ the partial resolution of singularities of $M$ obtained by blowing up $\overline{\Sigma}$, and by $\overline{\Sigma}'$ the exceptional divisor of the blowup.
\begin{thm}\label{Mukai}
The variety $M'$ is an irreducible symplectic V-manifold whose singularities are 28 points of analytic type $(\mathbb{C}^4 / \left\{\pm1\right\},0)$.
Moreover there is a Mukai flop between $M'$ and $\mathcal{P}$, which is an isomorphism between $M'\setminus\Pi'$and $\mathcal{P}\setminus\Pi$, where $\Pi'$and $\Pi$ are Lagrangian subvarieties isomorphic to $\mathbb{P}^{2}$.
\end{thm}
\begin{proof}
See Corollary 5.7 of \cite{Markou}.
\end{proof}

\section{The (1,2)-polarized Prym surfaces}\label{Pol}



In this section $\overline{B_{0}}$ and $\overline{\Delta_{0}}$ are smooth quartics tangent to each other at eight points lying on a conic (we will denote by $U$ the set of such pairs). Moreover, we assume that the pairs $(\overline{B_{0}},\overline{\Delta_{0}})$ and $(\overline{\Delta_{0}},\overline{B_{0}})$ are in $\mathfrak{L}$. We can permute the roles of $\overline{\Delta_{0}}$ and $\overline{B_{0}}$ in the above construction. Namely, consider the double cover $\widetilde{\mu}:\widetilde{X}\rightarrow \mathbb{P}^{2}$ branched in $\overline{\Delta_{0}}$. Let $\widetilde{\Delta_{0}}=\widetilde{\mu}^{-1}(\overline{\Delta_{0}})$, and let $\widetilde{B_{0}}$ and $\widetilde{B_{0}'}$ be the two curves mapped to $\overline{B_{0}}$ by $\widetilde{\mu}$. We denote by $\widetilde{i}$ the involution of $\widetilde{X}$ induced by $\widetilde{\mu}$ which exchanges $\widetilde{B_{0}}$ and $\widetilde{B_{0}'}$.
Consider the double cover $\widetilde{\rho}:\widetilde{S}\rightarrow \widetilde{X}$ branched in $\widetilde{B_{0}}$ and set $\widetilde{B}=\rho^{-1}(\widetilde{B_{0}})$. And, denote by $\widetilde{\tau}$ the involution of $S$ induced by $\widetilde{\rho}$. We have the diagram

$$\xymatrix{
 & & \widetilde{\Delta_{0}} \ar@{=}[rr] \ar@{^{(}->}[d]& & \overline{\Delta_{0}}\ar@{^{(}->}[d]\\
 \widetilde{S} \ar@(ul,dl)[]_{\widetilde{\tau}}\ar[rr]_{\widetilde{\rho}}^{2:1} & & \widetilde{X} \ar[rr]_{\widetilde{\mu}}^{2:1} & & \mathbb{P}^{2} \\
 \widetilde{B} \ar@{=}[rr] \ar@{^{(}->}[u] & & \widetilde{B_{0}} \ar@{^{(}->}[u] & & }$$
similar to (1).


Like in Section \ref{Rappels}, for a generic line in $\mathbb{P}^{2}$, we denote $\widetilde{E_{t}}=\widetilde{\mu}^{-1}(t)$, $\widetilde{C_{t}}=\widetilde{\rho}^{-1}(\widetilde{E_{t}})$, $\widetilde{\mu_{t}}=\widetilde{\mu}_{|E_{t}}$,  $\widetilde{\rho_{t}}=\widetilde{\rho}_{|C_{t}}$ and $\widetilde{\tau_{t}}=\widetilde{\tau}_{|C_{t}}$. The generic curves $\widetilde{E_{t}}$ are elliptic and the curves $\widetilde{C_{t}}$ are of genus 3. Thus, we have the tower of double covers:

$$\widetilde{C_{t}}\stackrel{2:1}{\rightarrow}\widetilde{E_{t}}\stackrel{2:1}{\rightarrow}\mathbb{P}^{1}.$$
By Lemma \ref{prym}, $\Prym(\widetilde{C_{t}},\widetilde{\tau_{t}})$ is also a (1,2)-polarized Prym surface.
We denote by $\Prym(C_{t},\tau_{t})^{\vee}$ the dual of the polarized abelian variety 
$\Prym(C_{t},\tau_{t})$.
The answer to one of our questions is given by the following proposition.
\begin{prop}
For a generic line $t\in\mathbb{P}^{2}$, we have the isomorphism: 
$$\Prym(C_{t},\tau_{t})^{\vee}\simeq\Prym(\widetilde{C_{t}},\widetilde{\tau_{t}}).$$
\end{prop}
\begin{proof}~~\\
\begin{itemize}
\item Step 1: \emph{The curve bigonally related to $C_{t}$}

Starting with the tower $C_{t}\rightarrow E_{t}\rightarrow\mathbb{P}^{1}$, we will construct a curve $C_{t}^{\vee}$ whose points correspond to the different ways to lift the pairs $\mu_{t}^{-1}(p)$, for $p\in\mathbb{P}^{1}$, to a pair in $C_{t}$, i.e.,
$$C_{t}^{\vee}=\left\{\left.p+q\in\Div^{(2)}(C_{t}) \right| \left[\rho_{t}(p)+\rho_{t}(q)\right]=\left[\mu_{t}^{*}\mathscr{O}_{\mathbb{P}^{1}}(1)\right] \right\}.$$
We denote by $\tau_{t}^{\vee}$ the involution $C_{t}^{\vee}\rightarrow C_{t}^{\vee}$ sending a lift to its complement. We have $\tau_{t}^{\vee}=\tau^{*}_{t|\Div^{(2)}(C_{t})}$.
Let $E_{t}^{\vee}=C_{t}^{\vee}/\tau_{t}^{\vee}$. 
We also define the map $\mu_{t}^{\vee}: E_{t}^{\vee}\rightarrow \mathbb{P}^{1}$ which sends a lift of $\mu_{t}^{-1}(p)$ ($p\in\mathbb{P}^{1}$) to $p$. 

For a better understanding we draw a diagram.
Let $\alpha$ be a generic point in $\mathbb{P}^{1}$ (a point which is not a branch point of $\mu_{t}$ nor the image of a branch point of $\rho_{t})$, $\beta_{i}$, $i=1$ or $2$ its preimages under $\mu_{t}$ and $\gamma_{i,j}$, $(i,j)\in\left\{1,2\right\}^{2}$, the preimages of the $\beta_{i}$ under $\rho_{t}$, as shown in the diagram:
$$\xymatrix@R=0pt@C=3pt{
C_{t} \ar[rr]^{\rho_{t}} & & E_{t} \ar[rr]^{\mu_{t}} & & \mathbb{P}^{1} \\
\gamma_{1,1} \ar[rr] & & \beta_{1} \ar[rr] & & \alpha.\\
\gamma_{1,2}\ar[rru] & & & & \\
\gamma_{2,1} \ar[rr]& & \beta_{2}\ar[rruu] & & \\
\gamma_{2,2} \ar[rru] & & & &
}$$
This gives the following diagram for points in $C_{t}^{\vee}$ and $E_{t}^{\vee}$:
$$\xymatrix@R=0pt{
C_{t}^{\vee} \ar[r]^{\rho_{t}^{\vee}} & E_{t}^{\vee} \ar[r]^{\mu_{t}^{\vee}} & \mathbb{P}^{1} \\
\gamma_{1,1}+\gamma_{2,1} \ar[r] &\overline{\gamma_{1,1}+\gamma_{2,1}}= \overline{\gamma_{1,2}+\gamma_{2,2}}\ar[r] & \alpha.\\
\gamma_{1,2}+\gamma_{2,2} \ar[ru] & & \\
\gamma_{1,1}+\gamma_{2,2} \ar[r] & \overline{\gamma_{1,1}+\gamma_{2,2}}= \overline{\gamma_{1,2}+\gamma_{2,1}}\ar[ruu] & \\
\gamma_{1,2}+\gamma_{2,1} \ar[ru] & & 
}$$
Like $\rho_{t}$, $\mu_{t}$ the maps $\rho_{t}^{\vee}$, $\mu_{t}^{\vee}$ are double covers:
$$C_{t}^{\vee}\stackrel{2:1}{\rightarrow}E_{t}^{\vee}\stackrel{2:1}{\rightarrow}\mathbb{P}^{1}.\ \ \ \ \ \ \ \ \ (*)$$
Barth \cite{Barth} calls this way to obtain $C_{t}^{\vee}$ Pantazis's bigonal construction (see \cite{Panta}, p. 304).
\begin{prop}
The abelian varieties $\Prym(C_{t},\tau_{t})$ and $\Prym(C_{t}^{\vee},\tau_{t}^{\vee})$ are dual to each other in such a way that $C_{t}^{\vee}$ (resp. $C_{t}$) embeds in $\Prym(C_{t},\tau_{t})$ (resp. $\Prym(C_{t}^{\vee},\tau_{t}^{\vee})$) as a theta-divisor of a polarisation of type (1,2).
\end{prop}
\begin{proof}
See \cite{Panta} Proposition 1 Section 3 page 307.
\end{proof}
Now, we will show that $\Prym(C_{t}^{\vee},\tau_{t}^{\vee})$ and $\Prym(\widetilde{C_{t}},\widetilde{\tau_{t}})$ are isomorphic.
To this end, we will look what happens when $\alpha$ is a branch point. 
\item Step 2: \emph{The ramification of the double covers of the diagram (*)}

We will denote by $(a_{i})_{1\leq i\leq 4}$ the branch points of $\mu_{t}$, $(b_{i})_{1\leq i\leq 4}$ their preimages in $E_{t}$, $(e_{i})_{1\leq i\leq 4}$ the branch points of $\rho_{t}$, $(p_{i})_{1\leq i\leq 4}$ their images in $\mathbb{P}^{1}$, $(e_{i}')_{1\leq i\leq 4}$ the other preimages of the $p_{i}$ in $E_{t}$, $(c_{i})_{1\leq i\leq 4}$ the preimages of the $(e_{i})_{1\leq i\leq 4}$ in $C_{t}$, $(b_{i,j})_{1\leq i\leq 4,1\leq j\leq 2}$ the preimages of the $(b_{i})_{1\leq i\leq 4}$ in $C_{t}$ and $(c_{i j}')_{1\leq i\leq 4,1\leq j\leq 2}$ the preimages of the $(e_{i}')_{1\leq i\leq 4}$ in $C_{t}$, as in the following diagram: 
$$\xymatrix@R=0pt@C=3pt{
C_{t} \ar[rr]^{\rho_{t}} & & E_{t} \ar[rr]^{\mu_{t}} & & \mathbb{P}^{1} \\
b_{i,j} \ar[rr] & & b_{i} \ar[rr] & & a_{i}\\
c_{i} \ar[rr] & & e_{i} \ar[rr] & & p_{i}.\\
c'_{i,j} \ar[rr] & & e'_{i} \ar[rru]& & 
}$$
We have $\tau_{t}(b_{i,1})= b_{i,2}$, $\tau_{t}(c'_{i,1})= c'_{i,2}$ for $\mu_{t}^{-1}(a_{i})=\left\{b_{i}\right\}$ and $\rho_{t}^{-1}(\mu_{t}^{-1}(a_{i}))$ $=\left\{b_{i,1},b_{i,2}\right\}$, for all $i\in \left\{1,2,3,4\right\}$. 
So we see that the ramification points of $\rho_{t}^{\vee}$ are the pairs $b_{i,1}+b_{i,2}$ $1\leq i\leq 4$. We have also $\mu_{t}^{-1}(p_{i})=\left\{e_{i},e'_{i}\right\}$ and $\rho_{t}^{-1}(\mu_{t}^{-1}(p_{i}))=\left\{c_{i},c'_{i,1},c'_{i,2}\right\}$. The involution $\tau_{t}^{\vee}$ exchanges $c_{i}+c'_{i,1}$ for $c_{i}+c'_{i,2}$. And the classes $c_{i}^{\vee}$ of the pairs $c_{i}+c'_{i,1}$ and $c_{i}+c'_{i,2}$  in $E_{t}^{\vee}$, $1\leq i\leq 4$, are the ramification points of $\mu_{t}^{\vee}$. We show this in the diagram:

$$\xymatrix@R=0pt{
C_{t}^{\vee} \ar[r]^{\rho_{t}^{\vee}} & E_{t}^{\vee} \ar[r]^{\mu_{t}^{\vee}} & \mathbb{P}^{1} \\
c_{i}+c'_{i,j}\ar[r] & \overline{c_{i}+c'_{i,1}}=\overline{c_{i}+c'_{i,2}} \ar[r]& p_{i}\\
b_{i,1}+b_{i,2} \ar[r]& \overline{b_{i,1}+b_{i,2}} \ar[r] & a_{i}.\\
2b_{i,j} \ar[r] & \overline{2b_{i,1}}=\overline{2b_{i,2}} \ar[ru] & 
}$$

\item Step 3: \emph{Conclusion}

We see that the maps $\mu_{t}^{\vee}$ and $\widetilde{\mu_{t}}$ have the same branch points in $\mathbb{P}^{1}$ by Step 2. This gives an isomorphism $\varphi_{t}$ between $E_{t}^{\vee}$ and $\widetilde{E_{t}}$. Now, we want that $\varphi_{t}$ sends the branch points of $\rho_{t}^{\vee}$ to the branch points of $\widetilde{\rho_{t}}$ to build an isomorphism between $\widetilde{C_{t}}$ and $C_{t}^{\vee}$. 

To show this, we map $E_{t}^{\vee}$ into $\widetilde{X}$ by $\varphi_{t}$. 
By step 2, $\widetilde{\mu}$ sends the branch points of $\rho_{t}^{\vee}$ on $t\cap \overline{B_{0}}=\left\{a_{1},...,a_{4}\right\}$ for all $t\in \mathcal{U}:=\mathbb{P}^{2\vee}\setminus \overline{B_{0}}^{\vee}$.
The group $\pi_{1}(\mathcal{U})$ acts by monodromy on the four points $\left\{a_{1},...,a_{4}\right\}$. This action is transitive because of the irreducibility of $\overline{B_{0}}$.
If we assume for a moment that $k$ of the 4 branch points of $\rho_{t}^{\vee}$ are on $\widetilde{B_{0}}$, and $4-k$ on $\widetilde{B_{0}'}$, $1\leq k\leq3$, then we see that the image of $\pi_{1}(\mathcal{U})$ is contained in $\mathscr{S}_{k}\times\mathscr{S}_{4-k}\subset\mathscr{S}_{4}$, which contradicts the transitivity. Hence the branch points of $\rho_{t}^{\vee}$  are all on $\widetilde{B_{0}}$ or all on $\widetilde{B_{0}'}$.
If they are on $\widetilde{B_{0}'}$, we just need to compose $\varphi_{t}$ with $\widetilde{i}$ (the involution on $\widetilde{X}$ we have defined in the very beginning) to obtain an isomorphism between $E_{t}^{\vee}$ and $\widetilde{E_{t}}$ which sends the branch points of $\rho_{t}^{\vee}$ to the branch points of $\widetilde{\rho_{t}}$. Denote this isomorphism by $\varphi_{t}$. Then we obtain the commutative diagram

$$\xymatrix{
\widetilde{C_{t}} \eq[d]\ar[rr]^{\widetilde{\rho_{t}}}& &\widetilde{E_{t}}\eq[d]_{\varphi_{t}}\ar[rr]^{\widetilde{\mu_{t}}}& & \mathbb{P}^{1}\\
C_{t}^{\vee}\ar[rr]^{\rho_{t}^{\vee}}& & E_{t}^{\vee}\ar[rru]^{\mu_{t}^{\vee}}& & },$$
which implies $\Prym(C_{t}^{\vee},\tau_{t}^{\vee})\simeq\Prym(\widetilde{C_{t}},\widetilde{\tau_{t}})$.
\end{itemize}
\end{proof}

\section{Relation between $S$ and $\widetilde{S}$}
It is a natural question to know whether the two K3 surfaces $S$, $\widetilde{S}$ are isomorphic or not.
We are going to prove that the answer is no for generic $S$. The meaning of "generic" will be clear from what follows.
\subsection{Definition of $\mathfrak{M}_{r,a,\delta}$}
Before giving the definition of $\mathfrak{M}_{r,a,\delta}$, we need some notions and some notation about lattices.
For a lattice $M$, we will denote its rank by $r(M)$. The signature of $M$ will be denoted by $\sign{M}=(b^{+}(M),b^{-}(M))$.
A lattice $M$ is \textit{Lorentzian} if $\sign{M}=(1,r(M)-1)$. We will denote by $M^{\vee}$ the dual of $M$ and by $A_{M}=M^{\vee}/ M$ the discriminant group. An even lattice $M$ is 2-\textit{elementary} if there is an integer $a$ with $A_{M}\simeq(\mathbb{Z}/2\mathbb{Z})^{a}$; then we set $a(M)=\dim_{\mathbb{Z}/2\mathbb{Z}}A_{M}$. We also define $\delta(M)=0$ if $x^{2}\in \mathbb{Z}\ \forall x\in M^{\vee}$,
otherwise $\delta(M)=1$. The triple $(\sign(M),a(M),\delta(M))$ determines the isometry class of an indefinite even 2-elementary lattice M by Theorem 3.6.2 of \cite{Lattice}.

Let $S$ be a K3 surface equipped with an antisymplectic involution $\tau:S\rightarrow S$. Let $P=\Pic(S)^{\tau}$. Then $P$ is a primitive 2-elementary Lorentzian sublattice of $H^{2}(S,\mathbb{Z})$ endowed with the cup product (see for instance Lemma 1.3 of \cite{Yo}). Let $(r,a,\delta)$ be a triple of integers. A couple $(S,\tau)$ is called a \textit{2-elementary K3 surface of type $(r,a,\delta)$} if $(r(P),a(P),\delta(P))=(r,a,\delta)$. We denote by $\mathfrak{M}_{r,a,\delta}$ the moduli space of isomorphism classes of 2-elementary K3 surfaces of type $(r,a,\delta)$. 

For a K3 surface $S$, $H^{2}(S,\mathbb{Z})$ endowed with the cup-product pairing is isometric to the K3 lattice $L=E_{8}(-1)^{2}\oplus U^{3}$. An isometry of lattices $\alpha: H^{2}(S,\mathbb{Z})\cong L$ is called a \textit{marking} of $S$. The pair $(S,\alpha)$ is called a \textit{marked K3 surface}. Let $M\subset L$ be a primitive 2-elementary Lorentzian sublattice. Let $I_{M}$ be the involution on $M\oplus M^{\bot}$ defined by $$I_{M}(x,y)=(x,-y),\ \ \ \ \ \ \ \ (x,y)\in M\oplus M^{\bot}.$$
Then $I_{M}$ extends uniquely to an involution on $L$ by Corollary 1.5.2 of \cite{Lattice}. A K3 surface equipped with an anti-symplectic holomorphic involution $\tau: S\rightarrow S$ is called a \textit{2-elementary K3 surface of type M} if there exists a marking $\alpha$ of $S$ satisfying $$\tau^{*}=\alpha^{-1}\circ I_{M}\circ\alpha.$$
Such a marking will be called a \textit{M-marking} of $(S,\tau)$. We note that $\alpha((\Pic(S))^{\tau})$ $=M$.
Now we will show that a 2-elementary K3 surface of type $(r,a,\delta)$ and a 2-elementary K3 surface of type M where $(r(M),a(M),\delta(M))=(r,a,\delta)$, are equivalent notions. 
\begin{lemme}\label{lemme}
Let $\varphi:N_{1}\simeq N_{2}$ be an isometry between two 2-elementary sublattices of $L$. We assume that $\sign N_{1}=\sign N_{2}= (2,x)$ where $x$ is an integer, then we can extend $\varphi$ to an isometry of $L$.
\end{lemme}
\begin{proof}
We will use Corollary 1.5.2 of \cite{Lattice}. We start by showing an isometry between $N_{1}^{\bot}$ and $N_{2}^{\bot}$. We have $\sign(N_{1}^{\bot})=\sign(N_{2}^{\bot})=(1,19-x)$ (because $\sign(L)=(3,19)$).
Since $L$ is unimodular, we have an isomorphism $\gamma_{N_{1}}: A_{N_{1}}\rightarrow A_{N_{1}^{\bot}}$ with $q_{N_{1}^{\bot}}\circ\gamma_{N_{1}}=-q_{N_{1}}$ and an isomorphism $\gamma_{N_{2}}: A_{N_{2}}\rightarrow A_{N_{2}^{\bot}}$ with $q_{N_{2}^{\bot}}\circ\gamma_{N_{2}}=-q_{N_{2}}$ (see \cite{Lattice} section 1).
This implies $(a(N_{1}^{\bot}),\delta(N_{1}^{\bot}))=(a(N_{2}^{\bot}),\delta(N_{2}^{\bot}))$. Then, by Theorem 3.6.2 of \cite{Lattice} there is an isometry $\psi: N_{1}^{\bot}\rightarrow N_{2}^{\bot}$. 

One the other hand, $\varphi$ (resp. $\psi$) induces an isometry $\overline{\varphi}: A_{N_{1}}\rightarrow A_{N_{2}}$ (resp. $\overline{\psi}: A_{N_{1}^{\bot}}\rightarrow A_{N_{2}^{\bot}}$). Now we take the following composition:
$$\gamma_{N_{2}}\circ\overline{\varphi}\circ\gamma_{N_{1}}^{-1}\circ\overline{\psi^{-1}},$$
which is an auto-isometry of $A_{N_{2}^{\bot}}$. 
Since $N_{2}^{\bot}$ is a 2-elementary Lorentzian sublattice, Theorem 3.6.3 of \cite{Lattice} gives us an isometry $\chi\in \mathcal{O}(N_{2}^{\bot})$ with $\overline{\chi}=\gamma_{N_{2}}\circ\overline{\varphi}\circ\gamma_{N_{1}}^{-1}\circ\overline{\psi^{-1}}$.
Hence $$\gamma_{N_{2}}\circ\overline{\varphi}=\overline{\chi\circ\psi}\circ\gamma_{N_{1}}.$$
By Corollary 1.5.2 of \cite{Lattice}, $\varphi$ extends to an isometry of $L$.
\end{proof}

\textbf{Remark}: The same result holds if $\sign{N_{i}}=(1,x)$, as we are going to see in the proof of the next proposition.

\begin{prop}\label{truc}
A K3 surface is a 2-elementary K3 surface of type $(r,a,\delta)$ if and only if it is a 2-elementary K3 surface of type M for some primitive 2-elementary Lorentzian sublattice $M$ with $(r(M),a(M),\delta(M))=(r,a,\delta)$.
\end{prop}
\begin{proof}
It is obvious that a 2-elementary K3 surface $(S,\tau)$ of type $M$ with $(r(M),a(M),\delta(M))$ $=(r,a,\delta)$ belongs to $\mathfrak{M}_{r,a,\delta}$.
So, we will show the other implication. Let $M\subset L$ with $(r(M),a(M),\delta(M))=(r,a,\delta)$ and let $(S,\tau)\in\mathfrak{M}_{r,a,\delta}$, we will find a $M$-marking of $S$. Let $P=\Pic(S)^{\tau}$, we can assume that $P$ is a sublattice of $L$, it suffices to take its image  by some marking of $S$.
We have $(r(P),a(P),\delta(P))=(r,a,\delta)$, then by Theorem 3.6.2 of \cite{Lattice}, we have an isometry $\psi: P\rightarrow M$. Moreover we have $\sign(P^{\bot})=\sign(M^{\bot})$ and since $L$ is unimodular we have $(a(P^{\bot}),\delta(P^{\bot}))=(a(M^{\bot}),\delta(M^{\bot}))$.
Once more, it follows by Theorem 3.6.2 of \cite{Lattice} that there is an isometry $\varphi: P^{\bot}\rightarrow M^{\bot}$. 
By Lemma \ref{lemme} this isometry extends to an isometry of $L$. We denote by $\widetilde{\varphi}$ this isometry,
we have $\tau^{*}=\widetilde{\varphi}^{-1}\circ I_{M}\circ\widetilde{\varphi}$. Indeed, let $x\in H^{2}(S,\mathbb{Z})$, $x=\frac{p+t}{2}$
where $p\in P$ and $t\in P^{\bot}$, hence,
$$\widetilde{\varphi}^{-1}\circ I_{M}\circ\widetilde{\varphi}(\frac{p+t}{2})=\widetilde{\varphi}^{-1}\circ I_{M}(\frac{\widetilde{\varphi}(p)+\widetilde{\varphi}(t)}{2})=\widetilde{\varphi}^{-1}(\frac{\widetilde{\varphi}(p)-\widetilde{\varphi}(t)}{2})=\frac{p-t}{2}.$$
\end{proof}

The moduli space of 2-elementary K3 surface was introduced by Nikulin in \cite{Nikulin0}, see also \cite{Nikulin} and \cite{Yo} Section 1 for more details.

\subsection{A Torelli Theorem for 2-elementary K3 surfaces}
For a better understanding of this moduli space we will give a kind of Torelli theorem for it (see \cite{Yo} and \cite{Yoshikawa} for more details).
We need some more notation.

Let $(S,\alpha)$ be a marked K3 surface. Recall the definition of the period map for marked K3 surfaces: the period of $(S,\alpha)$ is defined to be 
$$\pi(S,\alpha):=[\alpha(\eta)]\in\mathbb{P}(L\otimes\mathbb{C}),\ \ \ \ \ \eta\in H^{0}(S,\omega_{S})\setminus\left\{0\right\}.$$
Let $\Lambda$ be a lattice of signature $(2,n)$. We define 
$$\Omega_{\Lambda}:=\left\{[x]\in\mathbb{P}(\Lambda\otimes\mathbb{C}); \left\langle x,x\right\rangle=0, \left\langle x,\overline{x}\right\rangle>0\right\}.$$
Let $\Delta_{\Lambda}:=\left\{x\in \Lambda; \left\langle x,x\right\rangle=-2\right\}$.
\newline
For $\lambda\in \Lambda\otimes\mathbb{R}$, set $H_{\lambda}:=\left\{[x]\in\Omega_{\Lambda};\left\langle x,\lambda\right\rangle=0\right\}$. We define the discriminant locus of $\Omega_{\Lambda}$ by 
$$\mathcal{D}_{\Lambda}:=\sum_{d\in \Delta_{\Lambda}/\pm1}H_{d}.$$
Assume that $\Lambda$ is a primitive 2-elementary sublattice of $L$, with $\Lambda^{\bot}$ Lorentzian. 
Then we set
$$\Gamma(\Lambda):=\left\{g\in\mathcal{O}(L),\ I_{\Lambda^{\bot}}g=gI_{\Lambda^{\bot}}\right\},$$
$$\Gamma_{\Lambda}:=\left\{g_{|\Lambda}\in\mathcal{O}(\Lambda);\ g\in\Gamma(\Lambda) \right\},$$
$$\Omega_{\Lambda}^{°}:=\Omega_{\Lambda}\setminus\mathcal{D}_{\Lambda},\ \ \ \ \ \ \ \mathcal{M}_{\Lambda}^{°}:=\Omega_{\Lambda}^{°}/\Gamma_{\Lambda}.$$
The following theorem, due to Yoshikawa (\cite{Yo} Theorem 1.8) can be thought of as a Torelli Theorem for 2-elementary K3 surface:
\begin{thm}
Via the period map, the analytic space $\mathcal{M}_{M^{\bot}}^{°}$ is a coarse moduli space of 2-elementary K3 surfaces of type $M$.
\end{thm}
\begin{proof}
The proof uses the classical Torelli Theorem for K3 surface (see \cite{PSS} and \cite{BR}) and results of Nikulin \cite{Nikulin0}.
\end{proof}
Next, Yoshikawa improves this result in \cite{Yoshikawa}, proving the following proposition (Proposition 11.2 in \cite{Yoshikawa}).
\begin{prop}
The following equality holds:
$$\Gamma_{M^{\bot}}=\mathcal{O}(M^{\bot}).$$
\end{prop}
\begin{proof}
The proof uses Theorem 3.6.3 and Corollary 1.5.2 of \cite{Lattice}. The idea is the same as in the proof of Lemma \ref{lemme} and Proposition \ref{truc}.
\end{proof}
We thus obtain the following result.
Let $M\subset L$ be a primitive 2-elementary Lorentzian sublattice with $(r(M),a(M),\delta(M))=(r,a,\delta)$.
Define the map:
$$\begin{array}{ccccc}
 & \mathfrak{M}_{r,a,\delta}&\rightarrow &\Omega_{M^{\bot}}^{°}/\mathcal{O}(M^{\bot})\\
\overline{\varpi}_{r,a,\delta}: & & & \\
 & (S,\tau)& \mapsto & \mathcal{O}(M^{\bot})\cdot\pi(S,\alpha),
\end{array}$$
where $\alpha$ is a M-marking of $S$.
\begin{cor}\label{Torelli}
The map $\overline{\varpi}_{r,a,\delta}$
is an isomorphism.
\end{cor}
\begin{proof}
See \cite{Yoshikawa} page 8.
\end{proof}
We define $$\mathfrak{O}_{r,a,\delta}=\overline{\varpi}_{r,a,\delta}^{-1}\left(\left\{\left.\mathcal{O}(M^{\bot})\cdot\eta\in\Omega_{M^{\bot}}^{°}/\mathcal{O}(M^{\bot})\right|\left\langle \eta,x\right\rangle\neq 0,\ \forall x\in M^{\bot}\setminus\left\{0\right\}\right\}\right).$$
It is the dense subset of $\mathfrak{M}_{r,a,\delta}$, the one which explains the meaning of "generic".
It has the following important property.
\begin{prop}
Let $M\subset L$ be a 2-elementary Lorentzian sublattice with $(r(M),a(M),\delta(M))=(r,a,\delta)$. If $(S,\tau)\in \mathfrak{O}_{r,a,\delta}$, then 
a $M$-marking of $(S,\tau)$ induces an isometry between $\Pic S$ and $M$. In particular, all elements of $\Pic S$ are invariant by $\tau^{*}$.
\end{prop}
\begin{proof}
Indeed, let $x\in \Pic S$ then $\left\langle x,\eta\right\rangle=0$, where $\eta\in H^{0}(S,\omega_{S})\setminus\left\{0\right\}$.
Let $\alpha$ be a $M$-marking of $S$, then
we can write $\alpha(x)=\frac{a+b}{2}$, where $a\in M$ and $b\in M^{\bot}$. Since $\left\langle a, \alpha(\eta)\right\rangle=0$, we have $\left\langle b, \alpha(\eta)\right\rangle=0$, so by hypothesis $b=0$.
\end{proof}
\begin{cor}\label{involution}
Let $(S,\tau)$ and $(S',\tau')$ be in $\mathfrak{O}_{r,a,\delta}$, if $S$ and $S'$ are isomorphic then $(S,\tau)$ and $(S',\tau')$ are isomorphic.
\end{cor}
\begin{proof}
Let $M\subset L$ be a sublattice with $(r(M),a(M),\delta(M))=(r,a,\delta)$, let $\alpha_{\tau}$ and $\alpha_{\tau'}$ be $M$-markings of $(S,\tau)$ and $(S',\tau')$ respectively. Let $\eta\in H^{0}(S,\omega_{S})\setminus\left\{0\right\}$ and $\eta'\in H^{0}(S',\omega_{S'})\setminus\left\{0\right\}$.
We have the following digram:
$$\xymatrix{
H^{2}(S,\mathbb{Z}) \ar[dd]_{\varphi} \ar[rd]^{\alpha_{\tau}}& \\
& L \\
H^{2}(S',\mathbb{Z})\ar[ru]^{\alpha_{\tau'}}
 }$$
where $\varphi$ is a Hodge isometry.
Then we have $$\alpha_{\tau}(\eta)=(\alpha_{\tau}\circ\varphi^{-1}\circ\alpha_{\tau'}^{-1})_{|M^{\bot}}(\alpha_{\tau'}(\eta')).$$
Since $(\alpha_{\tau}\circ\varphi^{-1}\circ\alpha_{\tau'}^{-1})_{|M^{\bot}}\in\mathcal{O}(M^{\bot})$, we have $$\overline{\varpi}_{M}(S,\tau)=\overline{\varpi}_{M}(S',\tau').$$
\end{proof}

\subsection{Applications}\label{isomorphic}
Now, we will work with the moduli space $\mathfrak{M}_{8,8,1}$. 
\\

\textbf{Remark}: We have $M=I_{1,7}(2)$ for the associated Lorentzian sublattice of $L$, where $I_{p,q}$ stands for the lattice $\mathbb{Z}^{p+q}$ with quadratic form given by the diagonal matrix $$\diag(\underbrace{1,...,1}_{p},\underbrace{-1,...,-1}_{q})$$ and $\Lambda(d)$ denotes $\Lambda$ with quadratic form multiplied by $d$ for any lattice $\Lambda$ and any integer $d$.

We recall that $U$ is the locus of pairs $(\overline{B_{0}},\overline{\Delta_{0}})$ in $|\mathcal{O}_{\mathbb{P}^{2}}(4)|\times|\mathcal{O}_{\mathbb{P}^{2}}(4)|$ such that $\overline{B_{0}}$ and $\overline{\Delta_{0}}$ are smooth quartics, tangent to each other at eight points lying on a conic. We will denote by $\mu_{\overline{B_{0}}}:X_{\overline{B_{0}}}\rightarrow\mathbb{P}^{2}$ the double cover of $\mathbb{P}^{2}$ branched over $\overline{B_{0}}$.
We define $\mathcal{Q}\subset|\mathcal{O}_{\mathbb{P}^{2}}(4)|\times|\mathcal{O}_{\mathbb{P}^{2}}(4)|$ to be the set of pair $(\overline{B_{0}},2Q)$, where $\overline{B_{0}}$ is a smooth quartic and $Q$ a conic such that $\mu_{\overline{B_{0}}}^{*}(Q)$ is smooth. 

\begin{prop}
There is an isomorphism between $U/PGL_{3}\cup\mathcal{Q}/PGL_{3}$ and $\mathfrak{M}_{8,8,1}$. 
\end{prop}
\begin{proof}~~\\
\begin{itemize}
\item Step 1: \emph{The map $\mathscr{P}:U/PGL_{3}\cup\mathcal{Q}/PGL_{3}\rightarrow\mathfrak{M}_{8,8,1}$}

We will build $U/PGL_{3}\rightarrow\mathfrak{M}_{8,8,1}$; the construction of $\mathcal{Q}/PGL_{3}\rightarrow\mathfrak{M}_{8,8,1}$ is the same.

First, we have the map $U\rightarrow\mathfrak{M}_{8,8,1}$. 
Remember that diagram (1) of Section 1 gave a K3 surface $S$ with an involution $\tau$. By page 663 of \cite{Nikulin} $(S,\tau)\in\mathfrak{M}_{8,8,1}$. So the diagram (1) gives us the map $U\rightarrow\mathfrak{M}_{8,8,1}$. 
Now, let $(\overline{B_{0}},\overline{\Delta_{0}})$ and $(\overline{B_{0}}',\overline{\Delta_{0}}')$ be in $|\mathcal{O}_{\mathbb{P}^{2}}(4)|\times|\mathcal{O}_{\mathbb{P}^{2}}(4)|$ such that $f(\overline{B_{0}},\overline{\Delta_{0}})=(\overline{B_{0}}',\overline{\Delta_{0}}')$, where $f\in PGL_{3}$.
We can draw the following commutative diagram
$$\xymatrix@R=10pt{
B_{0}\ar@{^{(}->}[d]\ar@{=}[r]&\overline{B_{0}}\ar@{^{(}->}[d]\\
X\ar@(ul,dl)[]_{i}\ar[r]^{\mu}\eq[d]^{j}&\mathbb{P}^{2}\eq[d]^{f}\\
X'\ar@(ul,dl)[]_{i'}\ar[r]_{\mu'}&\mathbb{P}^{2}\\
B_{0}'\ar@{^{(}->}[u]\ar@{=}[r]&\overline{B_{0}}',\ar@{^{(}->}[u]
}$$
where $j$ is induced by $f$ (the other symbols are the same as in the diagram (1) of Section 1).
The map $j$ sends $\mu^{-1}(\overline{\Delta_{0}})=\Delta_{0}+i(\Delta_{0})$ on $\mu'^{-1}(\overline{\Delta_{0}}')=\Delta_{0}'+i'(\Delta_{0}')$. 
The curves $\overline{\Delta_{0}}$ and $\overline{\Delta_{0}}'$ are smooth, so they are irreducible. Then all the curves $\Delta_{0}$, $i(\Delta_{0})$, $\Delta_{0}'$ and $i'(\Delta_{0}')$ are irreducible.
Therefore $j$ sends $\Delta_{0}$ on $\Delta_{0}'$ or on $i'(\Delta_{0}')$. If $j$ sends $\Delta_{0}$ on $i'(\Delta_{0}')$, we replace $j$ by $i'\circ j$. Now, let $\rho:S\rightarrow X$ and $\rho':S'\rightarrow X'$ be the double covers branched in $\Delta_{0}$ and $\Delta_{0}'$ respectively. We get the following commutative diagram
$$\xymatrix@R=10pt{
 \Delta\ar@{^{(}->}[d]\ar@{=}[r]&\Delta_{0}\ar@{^{(}->}[d]\\
 S\ar@(ul,dl)[]_{\tau}\ar[r]^{\rho}\eq[d]^{\varphi}&X\eq[d]^{j}\\
 S'\ar@(ul,dl)[]_{\tau'}\ar[r]_{\rho'}&X'\\
 \Delta'\ar@{^{(}->}[u]\ar@{=}[r]&\Delta_{0}',\ar@{^{(}->}[u]
}$$
where $\varphi$ is induced by $j$. This implies $(S,\tau)\simeq(S',\tau')$, and we get the map: $$U/PGL_{3}\rightarrow \mathfrak{M}_{8,8,1}.$$
\item Step 2: \emph{The inverse function $\mathscr{G}:\mathfrak{M}_{8,8,1}\rightarrow U/PGL_{3}\cup\mathcal{Q}/PGL_{3}$}

Let $(S,\tau)$ be in $\mathfrak{M}_{8,8,1}$. By \cite{Nikulin}, $\rho:S\rightarrow X=S/\tau$ is a double cover ramified in a smooth curve of genus 3, $\Delta$ and $X$ is a del Pezzo surface. Moreover the linear system $|-K_{X}|$ induces a double cover $\mu:X\rightarrow \mathbb{P}^{2}$ branched in a smooth quartic of $\mathbb{P}^{2}$, $\overline{B}_{0}$. We have $\rho(\Delta)\in|-2K_{X}|$, then by Lemma 5.14 of \cite{Shouhei}, we have $(\overline{B}_{0},\mu(\rho(\Delta)))\in U$ or $(\overline{B}_{0},\mu(\rho(\Delta)))\in\mathcal{Q}$.
Now let $(S,\tau)$ and $(S',\tau')$ be two isomorphic objects from $\mathfrak{M}_{8,8,1}$. We denote by $(\overline{B}_{0},\overline{\Delta}_{0})$ and $(\overline{B}_{0}',\overline{\Delta}_{0}')$ the two pairs corresponding to $(S,\tau)$ and $(S',\tau')$ respectively (here $\overline{\Delta}_{0}$ and $\overline{\Delta}_{0}'$ may be a double conic).
To have a well defined map from $\mathfrak{M}_{8,8,1}$ to $U/PGL_{3}\cup\mathcal{Q}/PGL_{3}$, we must verify that $(\overline{B}_{0},\overline{\Delta}_{0})$ and $(\overline{B}_{0}',\overline{\Delta}_{0}')$ are exchanged by an automorphism of $\mathbb{P}^{2}$.
We have an isomorphism 
$f:S\simeq S'$ with $f\circ \tau=\tau'\circ f$. It induces a commutative diagram
$$\xymatrix{
S \eq[d]_{f}\ar[r]^{\rho}& X\eq[d]_{g}\ar[r]&|-K_{X}|\simeq\mathbb{P}^{2}\eq[d]_{(g^{-1})^{*}}\\
S'\ar[r]_{\rho'}& X'\ar[r]&|-K_{X'}|\simeq\mathbb{P}^{2}, }$$
which implies the result.
\end{itemize}
To finish, we see easily that the composition of $\mathscr{G}$ and $\mathscr{P}$ is the identity.
\end{proof}
\begin{cor}
The involution on $U/PGL_{3}$ given by $(\overline{B_{0}},\overline{\Delta_{0}})\rightarrow(\overline{\Delta_{0}},\overline{B_{0}})$ induces a rational involution of $\mathfrak{M}_{8,8,1}$ with indeterminacy on $\mathscr{P}(\mathcal{Q}/PGL_{3})$, which exchanges the two 2-elementary K3 surfaces $\mathscr{P}(\overline{B_{0}},\overline{\Delta_{0}})=(S,\tau)$ and $\mathscr{P}(\overline{\Delta_{0}},\overline{B_{0}})=(\widetilde{S},\widetilde{\tau})$. Moreover $(S,\tau)$ and $(\widetilde{S},\widetilde{\tau})$ are isomorphic if and only if there exists an automorphism $f$ of $\mathbb{P}^{2}$ such that $f(\overline{B_{0}},\overline{\Delta_{0}})=(\overline{\Delta_{0}},\overline{B_{0}})$.
\end{cor}


We define an open subset of $U/PGL_{3}$ by $$\mathscr{O}=\left\{\left.PGL_{3}\cdot(\overline{B_{0}},\overline{\Delta_{0}})\in U/PGL_{3}\right|PGL_{3}\cdot(\overline{B_{0}},\overline{\Delta_{0}})\neq PGL_{3}\cdot(\overline{\Delta_{0}},\overline{B_{0}})\right\}.$$
Now we are able to answer to the question we asked in the beginning of the section:
\begin{cor}\label{iso}
Let $(S,\tau)\in \mathfrak{O}_{8,8,1}\cap\mathscr{P}(\mathscr{O})$. Then $S$ and $\widetilde{S}$ are not isomorphic.
\end{cor}
\begin{proof}
Since $(S,\tau)\in\mathscr{P}(\mathscr{O})$, $(S,\tau)$ and $(\widetilde{S},\widetilde{\tau})$ are not isomorphic. Moreover $(S,\tau)\in \mathfrak{O}_{8,8,1}$ therefore by Corollary \ref{involution}, $S$ and $\widetilde{S}$ are not isomorphic either.
\end{proof}
\textbf{Remark}: 
\begin{itemize}
\item[1)]
The dimension of $\mathfrak{M}_{8,8,1}$ is 12.
\item[2)]
Let $\mathcal{B}:=\left\{PGL_{3}\cdot(\Gamma,\Gamma)\in U/PGL_{3}\right\}$; we have $\mathcal{B}\subset (U/PGL_{3})\setminus \mathscr{O}$. Moreover $\dim \mathcal{B}=6$, and $\mathscr{P}(\mathcal{B})$ parametrized the quadruple covers of $\mathbb{P}^{2}$ branched in smooth quartics. We also have $\mathcal{B}\subset (U/PGL_{3})\setminus(\mathfrak{L}/PGL_{3})$, where $\mathfrak{L}$ is the set of sufficiently generic pairs $(\overline{B_{0}},\overline{\Delta_{0}})$, see Definition \ref{gene}.
\item[3)]
The quotient variety $\mathcal{Q}/PGL_{3}$ has dimension 11. 
\end{itemize} 
~~\\

In fact, we can say even more: $S$ and $\widetilde{S}$ are not even derived equivalent.
We will denote by $D^{b}(S)$ the derived category of coherent sheaves on $S$.
Let $T_{S}$ be the transcendental lattice of $S$, that is the orthogonal complement to $\Pic S$ in $H^{2}(S,\mathbb{Z})$.
By Theorem 4.2.4. of \cite{Orlov}, the categories $D^{b}(S)$ and $D^{b}(S')$ are equivalent as triangulated categories if and only if
there exists a Hodge isometry between $T_{S}$ and $T_{S'}$.
We have the following theorem.
\begin{thm}\label{derived}
Let $S$ and $S'$ be two K3 surfaces such that $D^{b}(S)$ and $D^{b}(S')$ are equivalent. If $T_{S}$ is a 2-elementary sublattice of $H^{2}(S,\mathbb{Z})$, then $S$ and $S'$ are isomorphic.
\end{thm}
\begin{proof}
Let $S$ and $S'$ be two K3 surfaces such that $D^{b}(S)$ and $D^{b}(S')$ are equivalent.
By Theorem 4.2.4. of \cite{Orlov} we have a Hodge isometry $\rho:T_{S}\rightarrow T_{S'}$.
Let $\alpha:H^{2}(S,\mathbb{Z})\cong L$ and $\beta:H^{2}(S',\mathbb{Z})\cong L$ be markings of $S$ and $S'$ respectively.
The lattices $\alpha(T_{S})$ and $\beta(T_{S'})$ are two 2-elementary sublattices of $L$ of signature $(2,x)$. So by Lemma \ref{lemme}, $\beta\circ\rho\circ\alpha^{-1}_{|\alpha(T_{S})}$ extends to an isometry of $L$, that we will denote by $\nu$.
Then $\beta^{-1}\circ\nu\circ\alpha:H^{2}(S,\mathbb{Z})\rightarrow H^{2}(S',\mathbb{Z})$ is a Hodge isometry, therefore by the Global Torelli Theorem for K3 surfaces (see for instance Chapter 10, Theorem 5.3. of \cite{HuybrechtsK3}), $S$ and $S'$ are isomorphic.
\end{proof}
\textbf{Remark}: For all 2-elementary K3 surfaces $(S,\tau)\in\mathfrak{O}_{r,a,\delta}$, $T_{S}$ is a 2-elementary sublattice of $H^{2}(S,\mathbb{Z})$.
\begin{cor}
Let $(S,\tau)\in \mathfrak{O}_{8,8,1}\cap\mathscr{P}(\mathscr{O})$, then $D^{b}(S)$ and $D^{b}(\widetilde{S})$ are not equivalent.
\end{cor}
\begin{proof}
Indeed, if $(S,\tau)\in \mathfrak{O}_{8,8,1}$ then $T_{S}$ is a 2-elementary lattice. Then if $D^{b}(S)$ and $D^{b}(\widetilde{S})$ were equivalent, then $S$ and $\widetilde{S}$ would be isomorphic, which is false by Corollary \ref{iso}
\end{proof}

\section{Non-equivalence of dual Relative Compactified Prymians}\label{fin}
We will need the following proposition:
\begin{prop}\label{Hilbert}
Let $S$ and $S'$ be two complex K3 surfaces. If $S^{[2]}$ and $S'^{[2]}$ are bimeromorphic, then $D^{b}(S)\sim D^{b}(S')$.
\end{prop}
\begin{proof}
By Lemma 2.6 of \cite{Huybrechts}, if $S^{[2]}$ and $S'^{[2]}$ are birational, there is a Hodge isometry $\Phi$ between $H^{2}(S^{[2]},\mathbb{Z})$ and $H^{2}(S'^{[2]},\mathbb{Z})$, where $H^{2}(S^{[2]},\mathbb{Z})$ and $H^{2}(S'^{[2]},\mathbb{Z})$ are endowed with the Beauville-Bogomolov form.
Moreover, by Section 6 and Lemma 1 Section 9 of \cite{Beauville} we have $$H^{2}(S^{[2]},\mathbb{Z})=i(H^{2}(S,\mathbb{Z}))\oplus^{\bot}\mathbb{Z}\delta_{S},$$
where $i:H^{2}(S,\mathbb{Z})\rightarrow H^{2}(S^{[2]},\mathbb{Z})$ is a Hodge isometry and $\delta_{S}=\frac{1}{2}\Delta_{S}$ with $\Delta_{S}$ the class of the diagonal.
This implies: $$\left\{\left.a\in H^{2}\left(S^{[2]},\mathbb{Z}\right) \right| \ B_{S}(a,i(\eta_{S}))\neq0\right\}=\left\{\left.i(b)\in H^{2}(S,\mathbb{Z})\right|b\in T_{S}\right\},$$
where $\eta_{S}\in H^{0}(S,\omega_{S})\setminus\left\{0\right\}$, $B_{S}$ is the Beauville-Bogomolov form of $H^{2}(S^{[2]},\mathbb{Z})$ and $T_{S}$ is the transcendental lattice of $S$.
We have the same results for $S'$, so $\Phi$ induces a Hodge isometry between $T_{S}$ and $T_{S'}$. Then by Theorem 4.2.4 of \cite{Orlov}, $S$ and $S'$ are derived equivalent.
\end{proof}
We will denote by $\mathcal{P}_{(S,\tau)}$ the relative compactified Prymian built from the pair $(S,\tau)\in\mathscr{P}(\mathfrak{L}/PGL_{3})$, (see Definition \ref{Prym}).
If $(S,\tau)$ and $(S',\tau')$ are two isomorphic 2-elementary K3 surfaces, we have naturally $\mathcal{P}_{(S,\tau)}$ and $\mathcal{P}_{(S',\tau')}$ isomorphic. 
Now, we can prove the following theorem:
\begin{thm}
Let $(S,\tau)\in\mathfrak{O}_{8,8,1}\cap\mathscr{P}(\mathfrak{L}/PGL_{3})$ and $(S',\tau')\in\mathscr{P}(\mathfrak{L}/PGL_{3})$ such that $\mathcal{P}_{(S,\tau)}$ and $\mathcal{P}_{(S',\tau')}$ are isomorphic, then $(S,\tau)$ and $(S',\tau)$ are isomorphic.
\end{thm}
\begin{proof}
We will denote by $M_{(S,\tau)}$ and $M'_{(S,\tau)}$ the varieties defined in Section \ref{Rappels}, which are the quotients of $S^{[2]}$ by the involution $\iota_{S}$ and the partial resolution of singularities of $M_{(S,\tau)}$ respectively. We denote by $M_{(S',\tau')}$ and $M'_{(S',\tau')}$ the same varieties with $(S',\tau')$ instead of $(S,\tau)$. By Theorem \ref{Mukai}, $M'_{(S,\tau)}$ is bimeromorphic to $\mathcal{P}_{(S,\tau)}$ and $M'_{(S'\tau')}$ is bimeromorphic to $\mathcal{P}_{(S'\tau')}$. Therefore $M'_{(S,\tau)}$ and $M'_{(S'\tau')}$ are bimeromorphic, then $M_{(S,\tau)}$ and $M_{(S',\tau')}$ are bimeromorphic, hence also $M_{(S,\tau)}\setminus \Sing M_{(S,\tau)}$ and $M_{(S',\tau')}\setminus \Sing M_{(S',\tau')}$, so $S^{[2]}$ and $S'^{[2]}$ are bimeromorphic. By Proposition \ref{Hilbert} we have $D^{b}(S)\sim D^{b}(S')$, so by Theorem \ref{derived}, $S$ and $S'$ are isomorphic, and by Corollary \ref{involution} we have $(S,\tau)$ and $(S',\tau')$ isomorphic.
\end{proof}
\begin{cor}
The dense set $\mathfrak{O}_{8,8,1}\cap\mathscr{P}(\mathfrak{L}/PGL_{3})$ of $\mathfrak{M}_{8,8,1}$ provides a 1-to-1 parametrization of the relative compactified Prymians defined in Definition \ref{Prym}, and
the non trivial rational involution on $\mathfrak{M}_{8,8,1}$ defined in Section \ref{isomorphic} induces a non trivial involution on the set of the relative compactified Prymians.
\end{cor}
\bibliographystyle{amssort}

\begin{thebibliography}{10}

\bibitem{Barth}
W.~Barth
\newblock {\em Abelian Surfaces with (1,2)-Polarization},
\newblock Advanced Studies in Pure Mathematics 10, 1987
\newblock Algebraic Geometry, Sendai, 1985, 41-84.

\bibitem{Beau}
A.~Beauville
\newblock {\em Some remarks on Kähler manifolds with $c_{1}=0$}.
\newblock Classification of algebraic and analytic manifolds
\newblock (Katata 1982), 1-26 Progr. Math., 39, Birkhäuser Boston, MA, 1983.

\bibitem{Beauville}
A.~Beauville
\newblock {\em Variétés kähleriennes dont la première classe de Chern est nulle},
\newblock J. Differential geometry,
\newblock 18 (1983) 755-782.

\bibitem{Bogo}
F.A.~Bogomolov
\newblock{\em On the decomposition of Kähler manifolds with trivial canonical class},
\newblock Math. USSR-Sb 
\newblock 22 (1974), 580-583.

\bibitem{BR}
D.~Burns, M.~Rapoport
\newblock{\em On the Torelli problems for Kählerian K3 surfaces},
\newblock Ann. Sci. $\acute{E}$c. Norm. Supér, 
\newblock IV. Sér. 8, 235-274 (1975).

\bibitem{Fujiki}
A.~Fujiki
\newblock {\em On Primitively Symplectic Compact Kähler V-manifolds of Dimension Four}.
\newblock Classification of algebraic and analytic manifolds 
\newblock (Katata, 1982), 71-250.

\bibitem{Huybrechts}
D.~Huybrechts
\newblock{\em Compact hyperkähler manifolds: basic results},
\newblock Invent. math. 
\newblock 135, 63-113 (1999).

\bibitem{HuybrechtsK3}
D.~Huybrechts
\newblock{\em Lectures on K3 surfaces},
\newblock http://www.math.uni-bonn.de/people/huybrech/K3Global.pdf


\bibitem{Kata}
\newblock{\em Open problems $\S$6, in Classification of algebraic and analytic manifolds,},
\newblock Katata Symposium Proceeding 1982 sponsored by Taniguchi Foundation, K. Ueno, ed. Prog. in Math.
\newblock 39, Birkhauser (1983).


\bibitem{Shouhei}
S.~Ma,
\newblock {\em Rationality of the moduli spaces of 2-elementary K3 surfaces},
\newblock arXiv:1110.5110v1 [math.AG] 24 Oct 2011.

\bibitem{Markou}
D.~Markushevich, A.S.~Tikhomirov,
\newblock {\em New symplectic V-manifolds of dimension four via the relative compactified Prymian},
\newblock International Journal of Mathematics, Vol. 18, No. 10 (2007) 1187-1224,
\newblock World Scientific Publishing Company.

\bibitem{Mongordi}
G.~Mongardi
\newblock {\em Symplectic involutions on deformations of $K3^{[2]}$},
\newblock Cent. Eur. J. Math. 
\newblock 10 (2012), no.4, 1472-1485.

\bibitem{Mo}
S.~Mori,
\newblock {\em Classification of higher dimensional varieties},
\newblock in Alg. Geom Bowdoin 1985, Proc. Symp. in Pure Math.
\newblock 46, (1), 269-332.

\bibitem{Lattice}
V.V.~Nikulin
\newblock{\em Integral symmetric bilinear forms and some of their applications},
\newblock Math. USSR Izv.
\newblock 14 (1980), 103-167.

\bibitem{Nikulin0}
V.V.~Nikulin
\newblock{\em Factor groups of groups of automorphisms of hyperbolic forms with respect to subgroups generated by 2-reflections},
\newblock J. Soviet Math. 22, 1401-1476 (1983).

\bibitem{Nikulin}
V.V.~Nikulin
\newblock{\em Discrete Reflection Groups in Lobachevsky Spaces and Algebraic Surfaces},
\newblock proceedings of the International Congress of Mathematicians
\newblock Berkeley, California, USA, 1986, 654-671.

\bibitem{Grady1}
K.~O'Grady
\newblock{\em Desingularized moduli spaces of sheaves on K3},
\newblock J. Reine Angew. Math. 
\newblock 512 (1999), 49-117.

\bibitem{Grady2}
K.~O'Grady
\newblock{\em A new six dimensional irreducible symplectic variety},
\newblock J. Algebraic Geometry
\newblock 12 (2003), 435-505.

\bibitem{Orlov}
D.O.~Orlov
\newblock{\em Derived categories of coherent sheaves and equivalences between them},
\newblock Uspekhi Mat. Nauk 58(2003), no.3(351), 89-172;
\newblock translation in Russian Math. Surveys 58(2003), no.3, 511-591.

\bibitem{Panta}
S.~Pantazis,
\newblock {\em Prym varieties and the geodesic flow on SO(n)},
\newblock Math. Ann., 273 (1986), 297-315.

\bibitem{PSS}
I.~Piatetsky-Shapiro, I.R.~Shafarevich,
\newblock {\em A Torelli theorem for algebraic surfaces of type K3},
\newblock Math. USSR Izv.
\newblock 35, 530-572 (1971).

\bibitem{Sawon}
J.~Sawon,
\newblock {\em Derived equivalence of holomorphic symplectic manifolds}.
\newblock Algebraic structures and moduli spaces, 193-211,
\newblock CRM Proc. Lecture Notes, 38, Amer. Math. Soc., Providence, RI, 2004.

\bibitem{Yo}
K.-I.~Yoshikawa
\newblock{\em K3 surfaces with involution, equivariant analytic torsion, and automorphic forms on the moduli space},
\newblock Invent. Math. 156 (2004), no. 1, 53-117.

\bibitem{Yoshikawa}
K.-I.~Yoshikawa
\newblock{\em K3 surfaces with involution, equivariant analytic torsion, and automorphic forms on the moduli space II},
\newblock arXiv:1007.2830, to appear in J.Reine.Angew.Math.
\\
\end{thebibliography}

Departement of Mathematics, Lille1 University, 59655 Villeneuve d'Ascq

E-mail address: \texttt{gregoire.menet@ed.univ-lille1.fr}

\end{document}